\documentclass[12pt,a4paper]{amsart}

\title[Multiplication in the Iwahori--Hecke Algebra]{On the Complexity of Multiplication in the Iwahori--Hecke Algebra of the Symmetric Group}

\author[A. Niemeyer]{Alice C. Niemeyer}
\author[G. Pfeiffer]{G\"otz Pfeiffer}
\author[C. Praeger]{Cheryl E. Praeger}

\address{ACN: Lehrstuhl B f\"ur Mathematik, RWTH Aachen,
  Pontdriesch 10-16, 52062
  Aachen, Germany}
\email{alice.niemeyer@mathb.rwth-aachen.de}
\address{CEP: Centre for the Mathematics of Symmetry and Computation,
  University of Western Australia, 35 Stirling Highway, Crawley, WA
  6009, Australia}
\email{cheryl.praeger@uwa.edu.au}
\address{G.P.: School of Mathematics, Statistics and Applied Mathematics,
  National University of Ireland, Galway,  University Road,
  Galway, Ireland}
\email{goetz.pfeiffer@nuigalway.ie}

\usepackage[euler-digits]{eulervm}
\usepackage[margin=1in,includeheadfoot]{geometry}
\usepackage{parskip}
\usepackage{amssymb}
\usepackage[colorlinks=true]{hyperref}
\usepackage{tikz}
\tikzset{every picture/.style={scale=0.35,nodes={rectangle,inner sep=0.9mm}}}

\newtheorem{Theorem}{Theorem}[section]
\newtheorem{Proposition}[Theorem]{Proposition}

\newtheorem{Lemma}[Theorem]{Lemma}
\theoremstyle{definition}
\newtheorem{definition}[Theorem]{Definition}
\newtheorem{Example}[Theorem]{Example}
\newtheorem{Remark}[Theorem]{Remark}

\newcommand{\Size}[1]{\left|#1\right|}

\newcommand{\Sym}{\mathsf{Sym}}
\newcommand{\DD}{\mathcal{D}}

\renewcommand{\emptyset}{\varnothing}

\def\clap#1{\hbox to 0pt{\hss#1\hss}}
\def\mathclap{\mathpalette\mathclapinternal}
\def\mathclapinternal#1#2{%
\clap{$\mathsurround=0pt#1{#2}$}}

\numberwithin{equation}{section}

\begin{document}

\begin{abstract}
  We present new efficient data structures
for elements of Coxeter groups of type $A_m$
and their associated Iwahori--Hecke algebras $H(A_m)$.
Usually, elements of $H(A_m)$ are represented as
simple coefficient list of length $M = (m+1)!$
with respect to the standard basis, indexed by the
elements of the Coxeter group.  In the new data structure,
elements of $H(A_m)$ are represented as nested coefficient lists.
While the cost of addition is the same in both data structures,
the new data structure leads to a huge improvement
in the cost of multiplication in~$H(A_m)$.
\end{abstract}



\maketitle

\section{Introduction}
\label{sec:introduction}

Let $Z$ be a commutative ring with one.
The Iwahori-Hecke algebra $H$ of a finite Coxeter group $W$
over the ring $Z$ relative to a parameter $q \in Z$
is a $Z$-free
$Z$-algebra with a basis in bijection with the elements of the group $W$,
and a multiplication that is determined by deforming the
multiplication in $W$; see Section~\ref{sec:hecke} for details.
For a general introduction to finite Coxeter groups and their
Iwahori--Hecke algebras we refer the reader to the textbook by Geck
and Pfeiffer~\cite{GePf2000}, for the particular case of the symmetric group
we refer to Mathas' book~\cite{Mathas}.
They have applications in various branches of
Mathematics, such as knot theory, or the representation theory of
groups of Lie type (see \cite{Geck} and its references).

We study the complexity, or cost, of multiplication in the
Iwahori-Hecke algebra $H = H(A_m)$ of type $A_m$, so that $W$ is the symmetric
group $\Sym_{m+1}$. Thus $H$ has dimension $M=(m+1)!.$ The standard
basis for $H$ is indexed by the elements of $W$ and, for example
in the CHEVIE package~\cite{chevie} for GAP3, general elements of $H$ are
represented as \emph{simple coefficient lists} of length $M$ in this basis.

We introduce new data structures.   First we represent elements of $W$
in a \emph{tower data-structure}  described in Section~\ref{sec:a}, in
which the set of the transpositions $(i,i+1)$  for $1 \le i \le m$ is
replaced  by a  certain set  of roughly  $m^2/2$ cycles.   Then, using
this, in Section~\ref{sec:hecke} we represent each element of $H$ as a
\emph{nested  coefficient list}.  At the  top level,  this is  a list  of
length $m$  with entries (indexed  by certain of these  cycles) being
nested coefficient  lists representing  elements in  the Iwahori-Hecke
algebra $H(A_{m-1})$. We also develop tools for multiplying two such
lists.

Using either of these data-structures for $H$, multiplying elements $h, h'$
of $H$ necessarily requires us to consider each of the $M^2$ pairs of
coefficients of $h, h'$. We compare the `worst-case cost' (complexity)
of multiplying two elements of $H$ using simple coefficient lists
and nested coefficient lists. The following
 theorem is proved in Section~\ref{sec:complexity}.

\begin{Theorem}\label{the:main}
Let $Z$, $H$,  $W$, $M$ and  $m$ be as above.
\begin{enumerate}
\item[(a)]
  The cost of multiplying two elements in
$H$ each represented as a coefficient list
over $Z$,
using~\eqref{alg:simple}, is at most
$ \frac{m^2 + m + 4}{2} M^2$ operations in~$Z$.
\item[(b)]
  The cost of multiplying two elements in $H$,
each represented as a nested coefficient list over
$H(A_{m-1})$, is at most  $(1 + e) M^2$ operations in~$Z$, where $e=\exp(1)$.
\end{enumerate}
\end{Theorem}

\subsection{The standard data structure in CHEVIE}

The Coxeter group  $W =\Sym_{m+1}$ acting on  $\{1,\ldots,m+1\}$ has a
set   $S$  of   distinguished  generating   involutions,  namely   the
transpositions $(i,  i{+}1)$ of  adjacent points, $i  = 1,  \dots, m$.
Each element  $w \in W$ can  be written as  a product $w =  r_1 \dotsm
r_k$ of  $k$ Coxeter  generators $r_i  \in S$, and  $k$ is  called the
\emph{length} of  $w$, if it is  minimal with this property.   In that
case this  expression of  $w$ as  word in the  generators is  called a
\emph{reduced expression}.   A general element  $h \in H$ is  a linear
combination  of  basis elements  $T_w$,  $w  \in  W$, where  $T_w$  is
regarded as a product $T = T_{r_1} \dotsm T_{r_k}$, corresponding to a
reduced expression of $w = r_1 \cdots r_k$.

The standard strategy in existing implementations of Hecke algebra
arithmetic, such as the GAP3 package CHEVIE~\cite{chevie}, regards
an element $H$ as a linear combination of basis elements $T_w$, $w \in
W$, and reduces the task of multiplying two  elements of $H$  to the
case where the second element is $T_{s}$ for some $s\in S$.
The underlying data structure resembles the simple coefficient list,
that we consider here.

\subsection{Nested coefficient lists}

An element of $W = \Sym_{m+1}$ can alternatively be decomposed in
terms of coset representatives along the chain
$\Sym_1 < \Sym_2 < \dots < \Sym_{m+1}$ of `parabolic' subgroups  of
$W$. Here $\Sym_j$ is the symmetric group on $\{1,\ldots, j \}$ fixing
each point $i$ with $j  < i \le m+1$, and we choose a set $X_j$ of
coset representatives of $\Sym_j$ in $\Sym_{j+1}$ to be a set of
cycles $a(j,l)$ for $0 \le l\le j$ (see Equation~\ref{eq:cycle-amk}). Each
$w\in W$ can then be written uniquely as
$w = x_1 \dotsm x_m$ where each $x_j = a(j,a_j) \in X_j$ for some
$a_j$.  It is sufficient to record the list of integers $a_j$ for $1
\le j \le m$ and this is the \emph{tower} of $w$, see
Definition~\ref{def:tower}. Various properties such as the
Coxeter length  or the descent set of $w$ can be extracted easily
from the tower.
Furthermore, it is possible to find the tower of $w^{-1}$ from the
tower of $w$ (see Proposition~\ref{pro:inverseA})
 and the tower of a product of two permutations from their towers (see
 Proposition~\ref{pro:pro}).

We exploit the tower data structure to represent elements of $H$ as
nested coefficient lists, with the `nesting' corresponding to the
Iwahori-Hecke subalgebras $H(A_j)$ for the subgroups
$W(A_j) = \Sym_{j+1}$. In this case the task of
 multiplying two  elements of $H$  is reduced to  the
case where the second element is $T_{x}$ with $x\in \bigcup X_j$.
Proposition~\ref{pro:HA} provides a
 formula for the  efficient evaluation of such  a  product.

It is certainly desirable to develop similar techniques for other
types of Coxeter groups.  However, new ideas are needed in order to
generalize this construction from type~$A$ to other types.

\section{The tower data-structure for the symmetric group}
\label{sec:a}

For $m =  1, 2, \dotsc$, we denote the Coxeter group of type $A_m$
by $W(A_m)$.  Here we identify the group $ W(A_m)$ with the
symmetric group $\Sym_{m+1}$ on the $m+1$ points $\{1, \dots, m+1\}$.
Denote $s_i = (i,i{+}1)$ and $S = \{s_i: i = 1, \dots, m\}$.  The
elements of $S$ are called the \emph{simple reflections} of $W$ and
$S$ is a Coxeter generating set of $W$.  Thus every element can be
written as a product of simple reflections, usually in many different
ways.  For $w \in W(A_m)$, denote the \emph{length} of $w$ by
$\ell(w)$, that is $\ell(w)$
is the minimal number $k$ such that $w = s_{i_k} \dotsm s_{i_k}$ for
simple reflections $s_{i_j} \in S$.
We inductively define a normal form for elements of $W$ as
follows.  We define, for $0 \leq k \leq m $, elements $a(m,k)$, where
$a(m,0) = 1$ and for $k \ge 1$,
\begin{align} \label{eq:amk}
  a(m, k)  & = s_m s_{m-1} \dotsm s_{m-k+1} \in W(A_m).
\end{align}
For each $k$, the length
\begin{align} \label{eq:len-amk}
  \ell(a(m, k)) &= k.
\end{align}
Note that for $k \ge 1$,
$a(m, k)$ is the $(k+1)$-cycle
\begin{align} \label{eq:cycle-amk}
  a(m, k) = (m{-}k{+}1, m{-}k{+}2, \dots, m{+}1) \in \Sym_{m+1}
\end{align}
and
the set
\begin{align}
X_m &= \{a(m, k) : 0 \leq k \leq m\}
\end{align}
is the set of minimal length right coset representatives of
$W(A_{m-1})$ in $W(A_m)$.  This means that
for each element $w \in W(A_m)$ there are unique
elements $u \in W(A_{m-1})$ and $x_m = a(m, a_m) \in X_m$ such that
$w = u \cdot a(m, a_m)$, where the explicit multiplication dot indicates that
this decomposition of $w$ is \emph{reduced}, i.e.,
$\ell(w) = \ell(u) + \ell(x_m) = \ell(u) + a_m$.

Repeated application of the above for $j = 1, \dots, m$
yields integers $a_j \in \{0, \dots, j\}$ such that $w$
is the reduced product
\begin{align} \label{eq:reducedA}
  w = x_1 \cdot x_2 \cdot \dotsc \cdot x_m
\end{align}
of the coset representatives $x_j = a(j, a_j) \in X_j$
with the property that
\begin{align} \label{eq:lenA}
  \ell(w) &= \ell(x_1) + \ell(x_2) + \dots + \ell(x_m) \\\notag
&= a_1 + a_2 + \dots + a_m.
\end{align}

\begin{definition}\label{def:tower}
For $w \in W(A_m)$ as in Equation~\ref{eq:lenA} we call the sequence
$$
 \tau(w) =  (a_1, a_2, \dots, a_m)
$$
the \emph{tower of $w$}.
\end{definition}
The sequence $\tau(w)$ of $m$ non-negative integers
determines the terms $x_i$ in
the representation of $w$ in Equation~\ref{eq:reducedA} and thus yields a
decomposition of $w \in W(A_m)$
as a product of coset representatives along the chain of
subgroups
$W(A_0) < W(A_1) < \dots < W(A_m)$.

  \begin{Example}\label{ex:inv}
As an example of our data structure we consider  for $m=9$
the permutation $w = ( 1, 8,10, 3)( 2, 4, 6, 7, 5) \in \Sym_{10}$.
Then
\begin{align}\label{eq:w}
w &= s_1\cdot s_2 s_1\cdot s_3\cdot s_4 s_3 s_2\cdot s_5\cdot
s_6 s_5 s_4\cdot s_8\cdot s_9 s_8 s_7 s_6 s_5 s_4 s_3\\\notag
&= a(1,1) \cdot a(2,2) \cdot a(3,1) \cdot a(4,3) \cdot a( 5,1) \cdot
a(6,3) \cdot a(7,0) \cdot a(8,1) \cdot a(9,7).
\end{align}
As a tower, $w$ can be expressed as  $\tau(w) = (1, 2, 1, 3, 1, 3, 0,
1, 7)$. We illustrate the tower of $w$ with the following \emph{tower diagram}.
\medskip
      
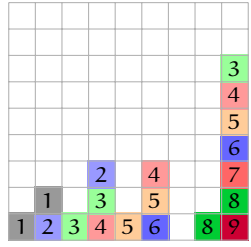
\begin{center}
  \begin{tikzpicture}[baseline]
\draw (0,0) node[fill=black!40] {$_1$};
\draw (1,0) node[fill=blue!40] {$_2$};
\draw (1,1) node[fill=black!40] {$_1$};
\draw (2,0) node[fill=green!40] {$_3$};
\draw (3,0) node[fill=red!40] {$_4$};
\draw (3,1) node[fill=green!40] {$_3$};
\draw (3,2) node[fill=blue!40] {$_2$};
\draw (4,0) node[fill=orange!40] {$_5$};
\draw (5,0) node[fill=blue!60] {$_6$};
\draw (5,1) node[fill=orange!40] {$_5$};
\draw (5,2) node[fill=red!40] {$_4$};
\draw (7,0) node[fill=green!80!blue] {$_8$};
\draw (8,0) node[fill=red!80!blue] {$_9$};
\draw (8,1) node[fill=green!80!blue] {$_8$};
\draw (8,2) node[fill=red!60] {$_7$};
\draw (8,3) node[fill=blue!60] {$_6$};
\draw (8,4) node[fill=orange!40] {$_5$};
\draw (8,5) node[fill=red!40] {$_4$};
\draw (8,6) node[fill=green!40] {$_3$};
\draw[shift={(-0.5,-0.5)},black!30](0,0) grid (9,9);
\end{tikzpicture}
      \end{center}
\medskip

The tower diagram for $w$ is constructed by writing the subscripts
of the letters of the $i$-th factor
$a(i,a_i)$ of $w$ into the $i$-th column of the diagram from the
bottom up. Thus  the tower diagram consists of columns of lengths
$\tau(w) = (1, 2, 1, 3, 1, 3, 0,
1, 7)$, that is the $i$-th column has height $a_i$.  Observe that the
position and the height of a column in the diagram uniquely determine
the actual entries of the column.
  \end{Example}

It is sometimes convenient to assume $a_0 = 0$ and $a_j = 0$ for $j
> m$, allowing us to view a permutation in  $S_m$ as an element of $S_j$
for any $j \ge m.$  It can also be convenient to omit trailing zeros from
$\tau(w)$.
From this representation
of a permutation $w$ one can read off quickly
\begin{itemize}
\item the length of $w$ as $\ell(w) = a_1 + a_2 + \dots + a_m$, by \eqref{eq:lenA};
\item a reduced expression for $w$ as
concatenation of the words
$a(1, a_1)$, $a(2, a_2)$,  \ldots,  $a(m, a_m)$, by \eqref{eq:reducedA};
\item the image of a point $i$ under the permutation $w$, or the
  permutation $w$ as product of the cycles $a(1, a_1)$, $a(2, a_2)$,
  \ldots, $a(m, a_m)$, using \eqref{eq:cycle-amk}.
\end{itemize}
We can also determine the \emph{descent set} $\DD(w) = \{s \in S: \ell(sw) < \ell(w)\}$ of $w$ from
its tower $\tau(w)$, as shown in Lemma~\ref{la:descentsA} below.

Furthermore, we can compute
\begin{itemize}
\item the tower $\tau(w w')$ of a product of $w, w' \in W(A_m)$ from the
  towers $\tau(w)$ and $\tau(w')$ (see Lemma~\ref{la:A});
\item the tower  $\tau(w^{-1})$ of the inverse of $w \in W(A_m)$
from the tower $\tau(w)$ (see Lemma~\ref{la:inverseA}).
\end{itemize}

From the permutation one can of course determine other properties, like the order of $w$,
its cycle structure and thus its conjugacy class in $W(A_m)$, etc.

In the next three subsections we consider the practicability of this
data structure for the purpose of extensive computations with
elements in $W(A_{m-1})$.  We will see
how to multiply two towers and how to invert a tower.

\subsection*{Inverses}
The tower $\tau(w^{-1})$ of the inverse of the permutation $w$ can be
computed from the tower $\tau(w)$ on the basis of the following
observation.

\begin{Lemma} \label{la:inverseA}
  Suppose $w = x_1 \dotsm x_m \in W(A_m)$ with $x_i =
  a(i, a_i)$ and $a_m \neq 0$.
Set $a_0 = 0$ and let $k$ be the largest
  index $i < m$ with $a_i = 0$. Moreover, let
\begin{align*}
  w' = x_1 \dotsm x_{k-1} \cdot x_k' \dotsm x_{m-1}' \in W(A_{m-1})
\end{align*}
with $x'_i = a(i,
  a_{i+1}-1)$, for $i \geq k$.  Then $w^{-1} = (w')^{-1} \cdot
  a(m, m-k)$.
\end{Lemma}

\begin{proof}
  By the hypothesis $a(k, a_k) = a(k, 0) = 1$ and $a_i > 0$ for all $i
  > k$.  Therefore we can write
\begin{align*}
  w &=
a(1, a_1)
\dotsm
a(k-1, a_{k-1})
\cdot
1
\cdot
a(k+1, a_{k+1})
\dotsm
a(m, a_m)
\\
&=
a(1, a_1)
\dotsm
a(k-1, a_{k-1})
\cdot
s_{k+1}
a(k, a_{k+1}-1)
\dotsm
s_m
a(m-1, a_m-1)
\\
&=
s_{k+1}
\dotsm
s_m \,
x_1
\dotsm
x_{k-1}
\cdot
x_k'
\dotsm
x'_{m-1}
\\
&=
 a(m, m -k)^{-1}
w',
\end{align*}
as required.
\end{proof}

Note that since $w' \in W(A_{m-1})$ its inverse $(w')^{-1}$ can now be computed recursively, by applying
Lemma~\ref{la:inverseA} to the smaller group $W(A_{m-1})$.

\begin{Proposition} \label{pro:inverseA}
Let $\tau(w) = (a_1, \ldots, a_m)$.
Let $k= \max\{i\mid 0\le i < m \mbox{ and } a_i = 0 \}.$
Then $\tau(w^{-1}) =(a_1',\ldots,a_m')$,
where $a_m' = a(m, m-k)$ and, recursively,
$(a_1',\ldots, a_{m-1}') = \tau((w')^{-1})$,
where $w' = a(m,m-k)\, w$ is a word of length
$\ell(w') = \ell(w) - (m{-}k)$.
\end{Proposition}

In practice, this process amounts to transposing and straightening the
tower diagram of $w$, as illustrated by the following example.

  \begin{Example}\label{ex:winv}
We continue Example~\ref{ex:inv} to find the tower diagram of the inverse of
the element $w$ given in Equation~\ref{eq:w}. The inverse $w^{-1}$ of $w$ is
\begin{eqnarray*}
  \lefteqn{s_3 s_4 s_5 s_6 s_7 s_8 s_9 \cdot s_8 \cdot s_4 s_5 s_6 \cdot
s_5 \cdot s_2 s_3 s_4 \cdot s_3 \cdot s_1 s_2 \cdot s_1}\\
&=& a(9,7)^{-1}  a(8,1)^{-1}  a(7,0)^{-1}  a(6,3)^{-1}
 a( 5,1)^{-1}
a(4,3)^{-1}  a(3,1)^{-1}  a(2,2)^{-1}  a(1,1)^{-1}.
\end{eqnarray*}
Observe that $w^{-1}$ can be represented by the transpose of the tower
diagram of $w$ as in
Figure~\ref{fig:inv}, that is to say,
  $a(9,7)^{-1}$ is  the bottom  row of
Figure~\ref{fig:inv}, and similarly for the other rows. The word for
$w^{-1}$ can be read off from Figure~\ref{fig:inv}, if we read
row by row from left to right and bottom to top.

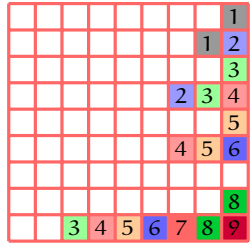
\begin{figure}
\begin{center}
  \begin{tikzpicture}[baseline]
\draw (8,8) node[fill=black!40] {$_1$};
\draw (8,7) node[fill=blue!40] {$_2$};
\draw (7,7) node[fill=black!40] {$_1$};
\draw (8,6) node[fill=green!40] {$_3$};
\draw (8,5) node[fill=red!40] {$_4$};
\draw (7,5) node[fill=green!40] {$_3$};
\draw (6,5) node[fill=blue!40] {$_2$};
\draw (8,4) node[fill=orange!40] {$_5$};
\draw (8,3) node[fill=blue!60] {$_6$};
\draw (7,3) node[fill=orange!40] {$_5$};
\draw (6,3) node[fill=red!40] {$_4$};
\draw (8,1) node[fill=green!80!blue] {$_8$};
\draw (8,0) node[fill=red!80!blue] {$_9$};
\draw (7,0) node[fill=green!80!blue] {$_8$};
\draw (6,0) node[fill=red!60] {$_7$};
\draw (5,0) node[fill=blue!60] {$_6$};
\draw (4,0) node[fill=orange!40] {$_5$};
\draw (3,0) node[fill=red!40] {$_4$};
\draw (2,0) node[fill=green!40] {$_3$};
\draw[shift={(-0.5,-0.5)},black!30](0,0) grid (9,9);
\draw [shift={(-0.5,-0.5)},red!60,very thick] (0,0) grid (9,9);
\end{tikzpicture}
\end{center}
\caption{Transpose of tower diagram for $w$ in Example~\ref{ex:inv}}\label{fig:inv}
\end{figure}

We find the tower diagram for $w^{-1}$ by applying
Lemma~\ref{la:inverseA} repeatedly. In the first step, applying
Lemma~\ref{la:inverseA}, we have
$m=9$
and the value of $k$ is
the largest integer with $a(k,0)^{-1}$ occurring in the above
expression, namely $k=7$. (It would be $0$ if all terms
$a(i,a_i)^{-1}$ had $a_i > 0$.) Note that $k=7$ is the largest integer
missing in the right most column of Figure~\ref{fig:inv}.
We replace $w^{-1}$ by $(w')^{-1} a(9,9-7),$ where
$w' =  a(1,1)\cdots a(6,3) \cdot x_7' \cdot x_8'$, where
$x_7' = a(7,a_8-1) = a(7,0)=1$ and $x_8' = a(8,a_9-1) = a(8, 6)
= s_8 s_7 s_6 s_5 s_4  s_3.$
Thus
$w^{-1} = \left( s_1 \cdot s_2 s_1 \cdot s_3 \cdot s_4 s_3 s_2
\cdot s_5 \cdot s_6 s_5 s_4 \cdot { s_8 s_7 s_6 s_5 s_4
  s_3}\right)^{-1} \cdot s_9 s_8.$ This expression is represented by
the left most diagram of Figure~\ref{fig:inv2}. Note that the $8\times 8$
subgrid with thick gridlines represents the transpose of the tower
diagram of $w'.$

Applying Lemma~\ref{la:inverseA} recursively to $w'$, we find the
following expressions for $w^{-1}$:  each line represents one
application of Lemma~\ref{la:inverseA}, where the $s_i$ in the terms $x_i=a(i,a_i)$
for $i\ge k$ have been underlined for emphasis.

\begin{eqnarray*}
w^{-1}
&=&  \left( s_1 \cdot s_2 s_1 \cdot s_3 \cdot s_4 s_3 s_2 \cdot s_5
\cdot s_6 s_5 s_4 \cdot {\color{red} \underline{s_8}} s_7 s_6 s_5 s_4 s_3 \right)^{-1}  \cdot s_9 s_8\\
&=&  \left( {\color{red} \underline{s_1}} \cdot {\color{red} \underline{s_2}} s_1 \cdot {\color{red} \underline{s_3}} \cdot {\color{red} \underline{s_4}} s_3 s_2 \cdot {\color{red} \underline{s_5}} \cdot {\color{red} \underline{s_6}} s_5 s_4 \cdot  {\color{red} \underline{s_7}} s_6 s_5 s_4 s_3\right)^{-1}  \cdot s_8 \cdot s_9 s_8\\
&=&  \left(  s_1 \cdot s_3 s_2 \cdot {\color{red} \underline{s_5}} s_4 \cdot   {\color{red} \underline{s_6}} s_5 s_4 s_3\right)^{-1}
\cdot s_7 s_6 s_5 s_4 s_3 s_2 s_1 \cdot s_8 \cdot s_9 s_8\\
&=&  \left(  s_1 \cdot {\color{red} \underline{s_3}} s_2 \cdot  {\color{red} \underline{s_4}} \cdot {\color{red} \underline{s_5}} s_4 s_3\right)^{-1}
\cdot s_6 s_5\cdot s_7 s_6 s_5 s_4 s_3 s_2 s_1 \cdot s_8 \cdot s_9 s_8\\
&=&  \left( s_1\cdot s_2 \cdot {\color{red} \underline{s_4}} s_3\right)^{-1}
\cdot s_5 s_4 s_3 \cdot s_6 s_5\cdot s_7 s_6 s_5 s_4 s_3 s_2 s_1 \cdot s_8 \cdot s_9 s_8\\
&=&   s_3    s_2 s_1 \cdot s_4 \cdot s_5 s_4 s_3 \cdot s_6 s_5\cdot s_7 s_6 s_5 s_4 s_3 s_2 s_1 \cdot s_8 \cdot s_9 s_8.
\end{eqnarray*}

This sequence of applications of Lemma~\ref{la:inverseA} is also
illustrated by the sequence of
diagrams in Figure~\ref{fig:inv2}. Note that in each diagram,
the value of $k$ used in Lemma~\ref{la:inverseA} is  given below.
Also, the  subdiagram
with thick gridlines represents the transpose of the tower diagram of
the element to which  Lemma~\ref{la:inverseA} is  applied.
Those  columns with thin gridlines represent the completed columns of
the tower diagram of $w^{-1}.$
Thus the tower of $w^{-1}$ is
 $\tau(w^{-1}) = (0, 0, 3, 1, 3, 2, 7, 1, 2)$.

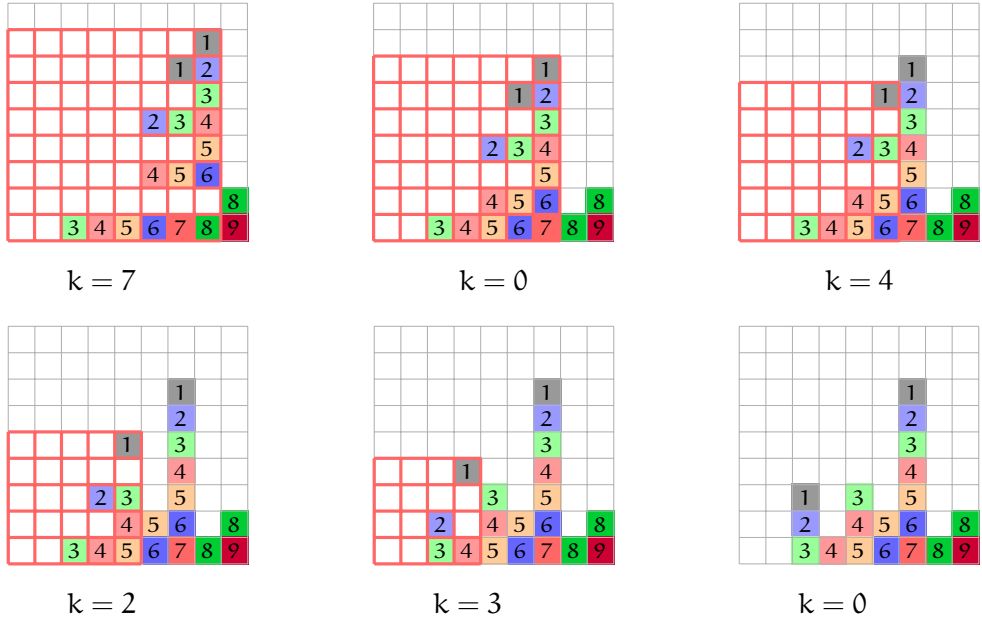
\begin{figure}[h]
\begin{center}
\hfill
  \begin{tikzpicture}[baseline]
\draw (7,7) node[fill=black!40] {$_1$};
\draw (7,6) node[fill=blue!40] {$_2$};
\draw (6,6) node[fill=black!40] {$_1$};
\draw (7,5) node[fill=green!40] {$_3$};
\draw (7,4) node[fill=red!40] {$_4$};
\draw (6,4) node[fill=green!40] {$_3$};
\draw (5,4) node[fill=blue!40] {$_2$};
\draw (7,3) node[fill=orange!40] {$_5$};
\draw (7,2) node[fill=blue!60] {$_6$};
\draw (6,2) node[fill=orange!40] {$_5$};
\draw (5,2) node[fill=red!40] {$_4$};
\draw (8,1) node[fill=green!80!blue] {$_8$};
\draw (8,0) node[fill=red!80!blue] {$_9$};
\draw (7,0) node[fill=green!80!blue] {$_8$};
\draw (6,0) node[fill=red!60] {$_7$};
\draw (5,0) node[fill=blue!60] {$_6$};
\draw (4,0) node[fill=orange!40] {$_5$};
\draw (3,0) node[fill=red!40] {$_4$};
\draw (2,0) node[fill=green!40] {$_3$};
\draw[shift={(-0.5,-0.5)},black!30](0,0) grid (9,9);
\draw [shift={(-0.5,-0.5)},red!60,very thick] (0,0) grid (8,8);
\draw (3,-2) node{{\small $k=7$}};
\end{tikzpicture}
\hfill
  \begin{tikzpicture}[baseline]
\draw (6,6) node[fill=black!40] {$_1$};
\draw (6,5) node[fill=blue!40] {$_2$};
\draw (5,5) node[fill=black!40] {$_1$};
\draw (6,4) node[fill=green!40] {$_3$};
\draw (6,3) node[fill=red!40] {$_4$};
\draw (5,3) node[fill=green!40] {$_3$};
\draw (4,3) node[fill=blue!40] {$_2$};
\draw (6,2) node[fill=orange!40] {$_5$};
\draw (6,1) node[fill=blue!60] {$_6$};
\draw (5,1) node[fill=orange!40] {$_5$};
\draw (4,1) node[fill=red!40] {$_4$};
\draw (8,1) node[fill=green!80!blue] {$_8$};
\draw (8,0) node[fill=red!80!blue] {$_9$};
\draw (7,0) node[fill=green!80!blue] {$_8$};
\draw (6,0) node[fill=red!60] {$_7$};
\draw (5,0) node[fill=blue!60] {$_6$};
\draw (4,0) node[fill=orange!40] {$_5$};
\draw (3,0) node[fill=red!40] {$_4$};
\draw (2,0) node[fill=green!40] {$_3$};
\draw[shift={(-0.5,-0.5)},black!30](0,0) grid (9,9);
\draw [shift={(-0.5,-0.5)},red!60,very thick] (0,0) grid (7,7);
\draw (4,-2) node {{\small $k=0$}};
\end{tikzpicture}
\hfill
  \begin{tikzpicture}[baseline]
\draw (6,6) node[fill=black!40] {$_1$};
\draw (6,5) node[fill=blue!40] {$_2$};
\draw (5,5) node[fill=black!40] {$_1$};
\draw (6,4) node[fill=green!40] {$_3$};
\draw (6,3) node[fill=red!40] {$_4$};
\draw (5,3) node[fill=green!40] {$_3$};
\draw (4,3) node[fill=blue!40] {$_2$};
\draw (6,2) node[fill=orange!40] {$_5$};
\draw (6,1) node[fill=blue!60] {$_6$};
\draw (5,1) node[fill=orange!40] {$_5$};
\draw (4,1) node[fill=red!40] {$_4$};
\draw (8,1) node[fill=green!80!blue] {$_8$};
\draw (8,0) node[fill=red!80!blue] {$_9$};
\draw (7,0) node[fill=green!80!blue] {$_8$};
\draw (6,0) node[fill=red!60] {$_7$};
\draw (5,0) node[fill=blue!60] {$_6$};
\draw (4,0) node[fill=orange!40] {$_5$};
\draw (3,0) node[fill=red!40] {$_4$};
\draw (2,0) node[fill=green!40] {$_3$};
\draw[shift={(-0.5,-0.5)},black!30](0,0) grid (9,9);
\draw [shift={(-0.5,-0.5)},red!60,very thick] (0,0) grid (6,6);
\draw (4,-2) node {{\small $k=4$}};
\end{tikzpicture}
\hfill\mbox{}

\bigskip
\hfill
  \begin{tikzpicture}[baseline]
\draw (6,6) node[fill=black!40] {$_1$};
\draw (6,5) node[fill=blue!40] {$_2$};
\draw (4,4) node[fill=black!40] {$_1$};
\draw (6,4) node[fill=green!40] {$_3$};
\draw (6,3) node[fill=red!40] {$_4$};
\draw (4,2) node[fill=green!40] {$_3$};
\draw (3,2) node[fill=blue!40] {$_2$};
\draw (6,2) node[fill=orange!40] {$_5$};
\draw (6,1) node[fill=blue!60] {$_6$};
\draw (5,1) node[fill=orange!40] {$_5$};
\draw (4,1) node[fill=red!40] {$_4$};
\draw (8,1) node[fill=green!80!blue] {$_8$};
\draw (8,0) node[fill=red!80!blue] {$_9$};
\draw (7,0) node[fill=green!80!blue] {$_8$};
\draw (6,0) node[fill=red!60] {$_7$};
\draw (5,0) node[fill=blue!60] {$_6$};
\draw (4,0) node[fill=orange!40] {$_5$};
\draw (3,0) node[fill=red!40] {$_4$};
\draw (2,0) node[fill=green!40] {$_3$};
\draw[shift={(-0.5,-0.5)},black!30](0,0) grid (9,9);
\draw [shift={(-0.5,-0.5)},red!60,very thick] (0,0) grid (5,5);
\draw (3,-2) node {{\small $k=2$}};
\end{tikzpicture}
\hfill
  \begin{tikzpicture}[baseline]
\draw (6,6) node[fill=black!40] {$_1$};
\draw (6,5) node[fill=blue!40] {$_2$};
\draw (3,3) node[fill=black!40] {$_1$};
\draw (6,4) node[fill=green!40] {$_3$};
\draw (6,3) node[fill=red!40] {$_4$};
\draw (4,2) node[fill=green!40] {$_3$};
\draw (2,1) node[fill=blue!40] {$_2$};
\draw (6,2) node[fill=orange!40] {$_5$};
\draw (6,1) node[fill=blue!60] {$_6$};
\draw (5,1) node[fill=orange!40] {$_5$};
\draw (4,1) node[fill=red!40] {$_4$};
\draw (8,1) node[fill=green!80!blue] {$_8$};
\draw (8,0) node[fill=red!80!blue] {$_9$};
\draw (7,0) node[fill=green!80!blue] {$_8$};
\draw (6,0) node[fill=red!60] {$_7$};
\draw (5,0) node[fill=blue!60] {$_6$};
\draw (4,0) node[fill=orange!40] {$_5$};
\draw (3,0) node[fill=red!40] {$_4$};
\draw (2,0) node[fill=green!40] {$_3$};
\draw[shift={(-0.5,-0.5)},black!30](0,0) grid (9,9);
\draw [shift={(-0.5,-0.5)},red!60,very thick] (0,0) grid (4,4);
\draw (3,-2) node {{\small $k=3$}};
\end{tikzpicture}
\hfill
\begin{tikzpicture}[baseline]
\draw (6,6) node[fill=black!40] {$_1$};
\draw (6,5) node[fill=blue!40] {$_2$};
\draw (2,2) node[fill=black!40] {$_1$};
\draw (6,4) node[fill=green!40] {$_3$};
\draw (6,3) node[fill=red!40] {$_4$};
\draw (4,2) node[fill=green!40] {$_3$};
\draw (2,1) node[fill=blue!40] {$_2$};
\draw (6,2) node[fill=orange!40] {$_5$};
\draw (6,1) node[fill=blue!60] {$_6$};
\draw (5,1) node[fill=orange!40] {$_5$};
\draw (4,1) node[fill=red!40] {$_4$};
\draw (8,1) node[fill=green!80!blue] {$_8$};
\draw (8,0) node[fill=red!80!blue] {$_9$};
\draw (7,0) node[fill=green!80!blue] {$_8$};
\draw (6,0) node[fill=red!60] {$_7$};
\draw (5,0) node[fill=blue!60] {$_6$};
\draw (4,0) node[fill=orange!40] {$_5$};
\draw (3,0) node[fill=red!40] {$_4$};
\draw (2,0) node[fill=green!40] {$_3$};
\draw[shift={(-0.5,-0.5)},black!30](0,0) grid (9,9);
\draw (3,-2) node {{\small $k=0$}};
\end{tikzpicture}
\hfill\mbox{}
      \end{center}
\caption{Recursive applications of Lemma~\ref{la:inverseA} for Example~\ref{ex:winv}}\label{fig:inv2}
\end{figure}
  \end{Example}

\subsection*{Products}
Multiplication of towers can be based on the following multiplicative
property of the coset representatives $a(m, k)$.  It allows a
product $a(m, k) \, a(j, l)$ with $j \leq m$ to be rewritten as a
product of at most two coset representatives in increasing order.

\begin{Proposition} \label{la:A}
  For $1 \leq j \leq m$, $0 \leq k \leq m$ and $0 \leq l \leq j$,
  \begin{align*}
    a(m, k) \, a(j, l) =
    \begin{cases}
      a(j, l) \cdot a(m, k), &  k < m - j, \\
      a(m, k{+}l), & k = m - j, \\
      a(j{-}1, l{-}1) \cdot a(m, k{-}1), & m - j < k \leq m-j + l, \\
      a(j{-}1, l) \cdot a(m, k), & k > m-j+l.
    \end{cases}
  \end{align*}
\end{Proposition}

If we write the factors as permutations
\begin{align*}
  a(m, k) &= (m{-}k{+}1, m{-}k{+}2, \dots, m{+}1), &
  a(j, l) &= (j{-}l{+}1, j{-}l{+}2, \dots, j{+}1),
\end{align*}
using
equation~\eqref{eq:cycle-amk},
with support
\begin{align*}
  M &= \{m{-}k{+}1, m{-}k{+}2, \dots, m{+}1\}, &
  J &= \{j{-}l{+}1, j{-}l{+}2, \dots, j{+}1\},
\end{align*}
respectively, and denote $J -1 = \{x-1 : x \in J\}$,
then the four cases of the lemma correspond to the cases
\begin{itemize}
\item $J \cap M = \emptyset$, the \emph{pass} case;
\item $J \cap M \neq \emptyset$ and $J - 1 \cap M = \emptyset$,
the \emph{join} case;
\item $\Size{J \cap M} > 1$ and $J - 1 \nsubseteq M$, the \emph{cancel} case;
\item $J \subseteq M$ and $J-1 \subseteq M$, the \emph{shift} case.
\end{itemize}

Note that, by equation~\eqref{eq:len-amk},  the length of the product is
\begin{align}
  \ell(a(m, k) \, a(j, l)) =
  \begin{cases}
    k + l - 2, & \text{if $m - j < k \leq m-j + l$,} \\
    k + l, & \text{otherwise,}
  \end{cases}
\end{align}
where $k + l$ is the sum of the lengths of the factors.  Hence
with the exception of the  \emph{cancel} case, all products
$a(m, k) \, a(j, l)$ are reduced.

\begin{proof}
Recall from equation~\eqref{eq:amk} that
\begin{align*}
  a(m, k) &= s_m s_{m-1} \dotsm s_{m-k+1}, &
  a(j, l) &= s_j s_{j-1} \dotsm s_{j-l+1}.
\end{align*}

In the \emph{pass} case, i.e., if $j < m - k$, all generators $s_i$
occurring in $a(j, l)$ commute with those in $a(m, k)$ and hence $a(m,
k) a(j, l) = a(j, l) \cdot a(m, k)$.

In the \emph{join} case, i.e., if $j = m - k$,
\begin{align*}
  a(m, k)\, a(j, l)
&= s_m s_{m-1} \dotsm s_{j+1} \cdot s_j s_{j-1} \dotsm s_{j-l+1}
= a(m, k+l).
\end{align*}

In the \emph{cancel} case, i.e., if $j > m-k$
and $j-l \leq m-k$,
the simple reflection $s_{m-k+1}$ occurs as a factor
in
\begin{align*}
  a(j, l) = s_j s_{j-1} \dotsm s_{m-k+2} \cdot s_{m-k+1}
\cdot s_{m-k} s_{m-k-1} \dotsm s_{j-l+1},
\end{align*}
and
\begin{align*}
   a(m,k)\, s_i &= s_{i-1} a(m, k), && \text{for $i \in \{ j, j-1, \dots, m-k+2\}$,} \\
   a(m, k)\, s_i & = a(m, k-1), && \text{for $i = m-k+1$, and} \\
   a(m, k-1)\, s_i & = s_i a(m,k-1), && \text{for $i \in \{ m-k, m-k-1, \dots, j-l+1\}$.}
\end{align*}
It follows that
\begin{align*}
  a(m, k)\, a(j, l) &= s_{j-1} s_{j-2} \dotsm s_{m-k+1} \cdot s_{m-k} s_{m-k-1} \dotsm s_{j-l+1} a(m, k-1) \\&= a(j-1, l-1) \cdot a(m, k-1),
\end{align*}
as desired.

Finally, in the \emph{shift} case, i.e., if $j > m - k + l$,
we have $a(m, k) s_i = s_{i-1} a(m, k)$
for all factors $s_i$ of $a(j, l) = s_j s_{j-1} \dotsm s_{j-l+1}$.
Hence
\begin{align*}
  a(m, k) a(j, l) &= s_{j-1} s_{j-2} \dotsm s_{j-l} a(m, k) = a(j-1, l) \cdot a(m, k),
\end{align*}
as desired.
\end{proof}

In particular, for $0 \leq k,l \leq m = j$, it follows that
\begin{align}\label{eq:squares}
  a(m,k)\, a(m,l) =
  \begin{cases}
    a(m,l), & k = 0,\\
    a(m{-}1,l{-}1) \cdot a(m,k{-}1), & 0 < k \leq l, \\
    a(m{-}1,l) \cdot a(m,k), & k > l. \\
  \end{cases}
\end{align}

As an immediate application,  we compute the descent set of  a
permutation $w$.

\begin{Lemma} \label{la:descentsA}
  Let $w \in W(A_m)$ with tower $\tau(w) = (a_1, \dots, a_m)$.  Then
  \begin{align*}
    \DD(w) = \{s_i : a_i > a_{i-1}\}.
  \end{align*}
  In particular, $s_i$ is the smallest descent of $w \neq 1$ if $i$ is
  the smallest index with $a_i > 0$.
\end{Lemma}
\begin{proof}
Write $w = w' a(i-1, a_{i-1}) a(i, a_i) w''$ and set $k = a_{i-1}$.
Let $1 \leq i \leq m$.   Note that $s_i a(j, a_j) = a(j, a_j) s_i$
for $j < i-1$, and that $s_i a(i-1, k) = a(i, k+1)$.  Hence
\begin{align*}
  s_i w = w'\, s_i\, a(i{-}1, a_{i-1})\, a(i, a_i)\, w''
 = w'\, a(i, a_{i-1}{+}1)\, a(i, a_i)\, w''.
\end{align*}
By equation~\eqref{eq:squares},
\begin{align*}
  a(i, a_{i-1}{+}1)\, a(i, a_i) =
  \begin{cases}
a(i{-}1, a_i{-}1)\, a(i, a_{i-1} ), &     a_{i-1} < a_i \\
a(i{-}1, a_i)\, a(i, a_{i-1}{+}1 ), &  a_{i-1} \geq a_i.\\
  \end{cases}
\end{align*}
It follows that $\ell(s_i w) < \ell(w)$ if and only if $a_{i-1} < a_i$,
as claimed.
\end{proof}

Proposition~\ref{la:A} implies that
for the tower $\tau(w) = (a_1, \dots, a_m)$, there is
a \emph{multiplication function}
\begin{align*}
  \mu_m \colon (j, k, l) \mapsto (j', k',l'),
\end{align*}
which
produces integers $j', k', l'$
from $j, k, l$ such that
$a(m, k)\, a(j,  l) = a(j',l')\, a(m, k')$.
More precisely, by
Proposition~\ref{la:A}
\begin{align}\label{eq:muA}
\mu_m(j, k, l)  =
\begin{cases}
  (j, k, l), & \text{if } j < m - k, \\
  (0, k+l, 0), & \text{if } j = m - k, \\
  (j-1, k-1, l-1), & \text{if }   m - k <  j \leq m - k + l, \\
  (j-1, k, l), & \text{if } j > m - k + l.
\end{cases}
\end{align}
This yields the following formula for
multiplying a tower $\tau(w) = (a_1, \dots, a_m)$
with a coset representative $a(j,l)$.
We denote the resulting tower $\tau(w \, a(j, l))$
by $(a_1, \dots, a_m) \star a(j, l)$.

\begin{Proposition}\label{pro:pro}
  Let $(a_1, \dots, a_m)$ be a tower and let $0 < l \leq j \leq m$.
  Then
  \begin{align*}
     (a_1, \dots, a_m) \star a(j, l) = (a_1', \dots, a_m'),
  \end{align*}
where $a_m'$ is determined by
$(j', a_m', l') = \mu_m(j, a_m, l)$
and, recursively,
\begin{align*}
  (a_1', \dots, a_{m-1}') =
\begin{cases}
(a_1, \dots, a_{m-1}), & \text{if } l' = 0, \\
(a_1, \dots, a_{m-1}) \star a(j',l'), & \text{if } l' > 0.
\end{cases}
\end{align*}
\end{Proposition}

\begin{proof}
If $w = a(1, a_1) \dotsm a(m, a_m)$ then, by the definition of $\mu_m$,
\begin{align*}
  w\, a(j, l)
&= a(1, a_1) \dotsm a(m{-}1, a_{m-1})\, a(m, a_m) \, a(j, l)
\\&= a(1, a_1) \dotsm a(m{-}1, a_{m-1})\, a(j', l')\, a(m, a_m'),
\end{align*}
where $(j', a_m', l') = \mu_m(j, a_m, l)$.
Now, either $l' = 0$, whence $a(j',l') = 1$, or $l'> 0$,
and
\begin{align*}
  \tau(a(1, a_1) \dotsm a(m{-}1, a_{m-1})\, a(j', l')) = (a_1, \dots, a_{m-1}) \star a(j',l'),
\end{align*}
by induction on $m$.
In any case, this yields
$
(a_1, \dots, a_m) \star a(j,l)
 =
(a_1', \dots, a_m')
$, as desired.
\end{proof}

The task of multiplying two towers $\tau(w)$ and $\tau(w')$
(that is to compute the tower $\tau(w w')$ from the towers
$\tau(w)$ and $\tau(w')$) is thus
reduced to incorporating the parts of $\tau(w')$ into $\tau(w)$, one
at a time:
\begin{align}\label{eq:productA}
  (a_1, \dots, a_m) \star (a_1', \dots, a_m') = \Bigl(\dots\bigl((a_1, \dots, a_m) \star a(1, a_1')\bigr) \star \dots\Bigr) \star a(m, a_m').
\end{align}

\section{Computing in the Iwahori--Hecke algebra}
\label{sec:hecke}

In this section we use the tower data structure to develop an
algorithm for the multiplication in the Iwahori--Hecke algebra of the
symmetric group $W=W(A_m)=\Sym_{m+1}$.

Let $Z$ be a commutative ring with one, and let $q \in Z$.
The Iwahori--Hecke algebra $H=H(A_m)$ of $W$ is the $Z$-free $Z$-algebra
(with one, denoted $1_H$) with basis $\{ T_w \mid w\in W\}$ and
multiplication defined by
\begin{align} \label{TwTs}
  T_w T_s =
  \begin{cases}
    T_{ws}, & \ell(ws) > \ell(w), \\
    (q-1) T_w + q T_{ws}, & \ell(ws) < \ell(w), \\
  \end{cases}
\end{align}
for $w \in W$, $s \in S$.

We first translate Lemma~\ref{la:A} into the context of Hecke algebras.

\begin{Lemma} \label{la:HA}
  For $m \geq j \geq 1$ and $k, l \geq 1$,
  \begin{align*}
    T_{a(m, k)} \, T_{a(j, l)} =
    \begin{cases}
      T_{a(j, l)} \, T_{a(m, k)}, &  k < m - j, \\
      T_{a(m, k{+}l)}, & k = m - j, \\
 (q-1) \, T_{a(j-1, j-m+k-1)} \, T_{a(m, m-j+l)} \\
\phantom{(q-1)(q-1)} +      q\, T_{a(j{-}1, l{-}1)} \, T_{a(m, k{-}1)},
& m - j < k \leq m-j + l, \\
      T_{a(j{-}1, l)} \, T_{a(m, k)}, & k > m-j+l.
    \end{cases}
  \end{align*}
\end{Lemma}

\begin{proof}
  In three of the four cases, the product $a(m,k)\, a(j,l)$ is reduced,
whence we get with Lemma~\ref{la:A} that
\begin{align*}
  T_{a(m, k)} \, T_{a(j, l)}
= T_{a(m,k)\, a(j,l)}
= T_{a(j',l')\, a(m,k')}
= T_{a(j',l')}\, T_{a(m,k')},
\end{align*}
where $(j', k', l') = \mu_m(j, k, l)$.
In the \emph{cancel} case, i.e., if $m-j < k \leq m-j+1$,
mimicking the proof of Lemma~\ref{la:A}, we get
\begin{align*}
    T_{a(m, k)} \, T_{a(j, l)}
&= T_{a(m, k)} \, (T_j T_{j-1} \dotsm T_{m-k+2}) \, T_{m-k+1} \,(T_{m-k} T_{m-k-1} \dotsm T_{j-l+1})
\\ &= (T_{j-1} T_{j-2} \dotsm T_{m-k+1}) \, T_{a(m, k-1)} \, T_{m-k+1}^2 \, (T_{m-k} T_{m-k-1} \dotsm T_{j-l+1})
\\ &= (q-1) \, T_{a(j-1, j-m+k-1)} \, T_{a(m, m-j+l)}  +      q\, T_{a(j{-}1, l{-}1)} \, T_{a(m, k{-}1)},
\end{align*}
as desired.
\end{proof}

A general element  $h \in H$ is a linear combination $h = \sum_{w \in W} z_w T_w$
with coefficients $z_w \in Z$. Moreover, it follows from
Equation~\ref{eq:reducedA}   and Lemma~\ref{la:HA} that we may also
write
\begin{align}
  h = \sum_{k = 0}^{m} h_k T_{a(m, k)}
\end{align}
with coefficients $h_k \in  H(A_{m-1})$, each of  which in turn is a
combination of the $T_{a(m-1,k)}$ with coefficients in $H(A_{m-2})$,
and so on. (Note that $H(A_0) = Z.$)

Thus we can represent an element $h\in H$
as a \emph{nested coefficient list}
$h=(h_0, \dots, h_m)$ with  $h_l\in H(A_{m-1})$, $l =0,\dots, m$.
It is obvious how to add two such elements.

Let $j \leq m$.
Assuming we can multiply
$h \in H(A_m)$
by  $g_l \in H(A_{j-1})$, $l = 0,\dots,j$,
we can multiply $h$ by $g = \sum_{l = 0}^{j} g_l T_{a(j, l)} \in H(A_j)$,
and obtain
\begin{align}
 h\,g = \sum_{l = 0}^{j} (h\,g_l) T_{a(j, l)},
\end{align}
where $h\,g_l \in H(A_m)$ is a $H(A_{m-1})$-combination of $T_{a(m, k)}$.
The products
$(h\,g_l) T_{a(j, l)}$ in turn can then
be evaluated using the following Proposition~\ref{pro:HA}
(where $h\,g_l$ is replaced by $h$).
This proposition  provides a recursive formula for the
computation of a product of the form $h T_{a(j, l)}$, where $h \in
H(A_m)$ and  $0 \leq l \leq j \leq m$.

\begin{Proposition}\label{pro:HA}
Let $h = (h_0,\ldots, h_m)\in H(A_m)$ and let $j,l$ be integers
such that $1 \le j \le m$ and   $0\le l \leq j.$
Then
\begin{align*}
   ( h_0,\ldots, h_m) T_{a(j, l)}
&= ( h_0',\ldots, h_m'),
\end{align*}
where $h_k'$ for $0\le k \le m$ is given in the following table:

\begin{center}
\begin{tabular}{cll}
Line & $h_k'$ & Conditions\\
\hline
$({\rm a})$ &    $ h_k T_{a(j,l)}$ & $k < m - j$ \\
$({\rm b})$ & $q h_{k+1} T_{a(j-1,l-1)}$ & $m-j \leq k < m-j+l $\\
$({\rm c})$ & $h_{m-j} + (q-1) \sum_{i=1}^{l} h_{m-j+i} T_{a(j-1, i-1)}$  & $k = m - j + l$ \\
$({\rm d})$ & $   h_k T_{a(j-1,l)}$ & $k > m - j + l $
\end{tabular}
\end{center}
\end{Proposition}

\begin{proof}
Note that $h= ( h_0,\ldots, h_m)  = \sum_{k = 0}^{m} h_k T_{a(m, k)}$
and we are computing $h' = ( h_0',\ldots, h_m')  =
\sum_{k = 0}^{m} h_k' T_{a(m, k)}$ such that
$h T_{a(j, l)} = h'$. If $l=0$ then $T_{a(j,l)} = 1_H$ and hence $h' =
h$ so $h_k' = h_k$ as in Case (a), (c) or (d), according as $k < m-j,
k= m-j,$ or $k > m-j$, respectively. Assume now that $l > 0.$

\medskip
Suppose first that  $k < m-j$ then
    $T_{a(m, k)} T_{a(j, l)} =
   T_{a(j, l)} T_{a(m, k)}$  and we see below that this is the only
   contribution
to $h_k'$. In particular, $h_k' = h_k T_{a(j, l)}$, and  Line (a) holds.

We defer the case $k=m - j$ and consider
next $k\in [m-j +1, m-j+l]$. Then
\begin{multline}\label{eq:HA3}
  \sum_{\mathclap{k = m-j+1}}^{\mathclap{m-j+l}} h_k T_{a(m, k)} T_{a(j, l)}\\
 = \sum_{\mathclap{k = m-j+1}}^{\mathclap{m-j+l}}
(q-1) h_k T_{a(j-1, j-m+k-1)} T_{a(m, m-j+l)}
+ q h_k T_{a(j-1,l-1)} T_{a(m, k-1)} \\
 = (q-1) \left(\sum_{\mathclap{i = 1}}^{\mathclap{l}}
 h_k T_{a(j-1, i-1)}\right) T_{a(m, m-j+l)}
+ \sum_{\mathclap{k = m-j}}^{\mathclap{m-j+l-1}} q h_{k+1} T_{a(j-1,l-1)} T_{a(m, k)}.
\end{multline}
If $m-j + 1 \le k < m-j+l$ this yields a contribution of
$ q h_{k+1} T_{a(j-1,l-1)}$ to $h_k'$. We again see below this is
the only contribution to $h_k'$ and Line (b) holds.

We find two contributions to $h_{m-j+l}'$. The first
contribution comes from the coefficient of $T_{a(m, m-j+l)}$ in
Equation~\eqref{eq:HA3}, namely
$(q-1) \sum_{i=1}^{l} h_{m-j+i} T_{a(j-1, i-1)}$.
The second contribution comes from considering $k = m-j$, where we find
  \begin{align*}
    h_k T_{a(m, k)} T_{a(j, l)} =
    h_k T_{a(m, k + l)} =  h_{m-j} T_{a(m, {m-j+l})},
  \end{align*}
that is a  contribution of  $h_{m-j}$.
We find no other contributions below, and hence the sum
of these two yields Line (c).

Finally, let $m-j+l < k \leq m$. Then
$T_{a(m, k)} T_{a(j, l)} = T_{a(j-1, l)} T_{a(m, k)}$
and hence
$  h_k' =  h_k T_{a(j-1, l)}$ and Line (d) holds.
\end{proof}

\begin{Example}\label{ex:product}
  We illustrate the process with the computation of the product $hg$
  of an element $h \in H(A_2)$ and an element $g \in H(A_1)$.  We have
  \begin{align*}
    h = h_0 T_{a(2,0)} + h_1 T_{a(2,1)} + h_2 T_{a(2,2)},
  \end{align*}
  where $h_0, h_1, h_2 \in H(A_1)$ and, for $i \in \{0, 1, 2\}$,
  \begin{align*}
    h_i = h_{i0} T_{a(1,0)} + h_{i1} T_{a(1,1)},
  \end{align*}
  with $h_{i0}, h_{i1} \in H(A_0) = Z$. All in all, $h \in H(A_2)$ is
  represented by $3$ elements $h_i \in H(A_1)$, or by $3! = 6$ scalars
  $h_{ij} \in H(A_0) = Z$.  We write
  \begin{align*}
    h = (h_0, h_1, h_2) = ((h_{00}, h_{01}), (h_{10}, h_{11}), (h_{20}, h_{21}))\text.
  \end{align*}
  Note that multiplying $h$ with a scalar $z \in Z$, or adding any two
  elements of $H(A_2)$, requires $6$ operations in $Z$.  In a similar
  way,
  \begin{align*}
    g = g_0 T_{a(1,0)} + g_1 T_{a(1,1)} = (g_0, g_1) \in H(A_1)\text,
  \end{align*}
  where $g_i \in H(A_0)$ for $i = 0, 1$.
Recall that $T_{a(j, 0)} = 1_H$, for $j \geq 0$, so that such factors in a product can simply be ignored.
 Then computing
\begin{align*}
  hg = h \, (g_0, g_1) = h\,(g_0 T_{a(1,0)} + g_1 T_{a(1,1)}) &= h g_0 + (h g_1) T_{a(1,1)}\text,
\end{align*}
requires $3 \cdot 6 = 18$ operations in $Z$ ($6$ each for
multiplying $h \in H(A_2)$ with the scalars
$g_i$, and another $6$ for adding two elements in $H(A_2)$),
plus the number of operations needed to determine
$h'T_{a(1,1)}$, where $h' = h g_1$.
We get
\begin{align*}
  h' T_{a(1,1)} = (h'_0, h'_1, h'_2) T_{a(1,1)} = (h''_0, h''_1, h''_2)\text,
\end{align*}
where $h''_k$, $k \in \{0, 1, 2\}$, are  determined by
Proposition 3.2
with $m = 2$ and $j = l = 1$, as follows.
For $k = 2$, case (c) yields
\begin{align*}
  h''_2 = h'_1 + (q{-}1) h'_2\text,
\end{align*}
using $4$ operations in $Z$, as each $h'_i \in H(A_1)$ is represented by $2$ scalars.
For $k = 1$, case (b) yields
\begin{align*}
  h''_1 = q h'_2\text,
\end{align*}
using $2$ operations in $Z$.
For $k = 0$, case (a) yields
\begin{align*}
  h''_0 &= h'_0 T_{a(1, 1)} = (h'_{00}, h'_{01}) T_{a(1,1)}
= (h''_{00}, h''_{01})\text,
\end{align*}
where $h''_{00}$ and $h''_{01}$ are  determined by
a further application of Proposition 3.2
with $m = 1$ and $j = l = 1$, as follows.
Case (c) yields
\begin{align*}
h''_{01} = h'_{00} + (q{-}1) h'_{01}\text,
\end{align*}
using $2$ operations in $Z$.
Case (b) yields
\begin{align*}
h''_{00} = q h'_{01}\text,
\end{align*}
using $1$ further operation in $Z$.

In total, the computation of $hg$ in terms of nested coefficient lists
needs $27$ operations in $Z$.  Note that computing the same product $hg$ in
terms of simple coefficient lists needs the same number of operations
in $Z$, as $T_{a(1,1)} = T_{s_1}$.
\end{Example}

We now determine the complexity of the arithmetical operations in
$H = H(A_m)$ in general.

\section{Complexity}
\label{sec:complexity}

We  compare the complexity of the arithmetical operations in $H = H(A_m)$,
using different data structures to represent elements of $H$,
once as simple coefficient list $(z_w)_{w \in W}$ over $Z$, once as
nested coefficient list $(h_0, \dots, h_m)$ over $H(A_{m-1})$.
In each case we cost addition and multiplication in $Z$ as $1$ unit.
Let
\begin{align*}
  M = \Size{W(A_m)} = \dim_Z H(A_m) = (m+1)!
\end{align*}
 The sum of two
elements $h, h' \in H(A_m)$ obviously needs $M$ operations in~$Z$,
in either representation.

\subsection{Simple coefficient lists}
We now determine the cost of multiplying two elements
in $H(A_m)$ where each is represented as a simple coefficient list
$(z_w)_{w \in W}$ over $Z$.
Consider first the product of a single basis element $T_w$, $w \in W$,
with a generator $T_s$, $s\in S$.  The relations of $H$ allow us to
compute their product in $H$ based on the formula~(\ref{TwTs}).

Now suppose that
\begin{align*}
  h = \sum_{w \in W} z_w T_w \in H(A_m),
\end{align*}
with coefficients $z_w \in Z$.
We represent $h$ as its list of coefficients $(z_w)_{w \in W}$.
Then
\begin{align}\label{eq:starTs}
  h T_s =
\sum_{w \in W} z'_w T_w \in H(A_m),
\end{align}
where
\begin{align*}
  z'_w =
  \begin{cases}
    q z_{ws}, & \ell(ws) > \ell(w), \\
    z_{ws} + (q-1) z_w, & \ell(ws) < \ell(w).
  \end{cases}
\end{align*}
On the basis of this, we can define an operation of
$T_s$, $s \in S$, on coefficient lists as
\begin{align}
(z_w)_{w\in W} \star T_s = (z'_w)_{w \in W}.
\end{align}
Then $(z_w)_{w\in W} \star T_s$ is the coefficient list of $h T_s$.

For $b \in Z$, we set $b(z_w)_{w\in W} = (b z_w)_{w\in W}$.

For an element $v = r_1 r_2 \dotsm r_l \in W$ of length $\ell(v) = l$, with $r_1, \dots, r_l \in S$,
we note that $h T_v = (\dots((h T_{r_1}) T_{r_2}) \dotsm T_{r_l})$
and set
\begin{align}\label{eq:starTv}
(z_w)_{w\in W} \star T_v = (\dots(((z_w)_{w\in W} \star T_{r_1})\star  T_{r_2})\star \dotsm\star T_{r_l}).
\end{align}
Then $(z_w)_{w\in W} \star T_v$ is the coefficient list of $h T_v$.

Finally, if
\begin{align*}
  h' = \sum_{v \in W} b_v T_v \in H(A_m),
\end{align*}
with coefficients $b_v \in Z$ we can compute the product $hh'$ recursively as
\begin{align*}
  hh' = \sum_{v \in W} h b_v T_v \in H(A_m),
\end{align*}
The coefficient list of $h h'$ is
\begin{align}\label{alg:simple}
  \sum_{v \in W} b_v (z_w)_{w\in W} \star T_v,
\end{align}
where addition of coefficient lists is the usual vector addition.

\begin{proof}[Proof of Theorem~\ref{the:main}(a)]
In order to determine the cost of a product, we proceed as follows.
We first consider the cost of computing $(z_w)_{w\in W} \star T_s = (z'_w)_{w\in W}$
using~\eqref{eq:starTs}.

Computing $z'_w$, $w \in W$, needs
at most $2$ operations in $Z$.
Hence, computing the coefficients $z_w'$ requires a total of $2 M$ operations.
Consequently, computing the coefficient list  $ b_v (z_w)_{w\in W} \star T_v$
for  $v \in W$
and a  coefficient $b_v \in Z$
using~\eqref{eq:starTv}
requires $(2 \ell(v)+1) M$ operations.
Finally, using the fact that
\begin{align*}
  \sum_{v \in W} \ell(v) = \frac M2 \binom{m+1}2,
\end{align*}
computing the coefficient list of the product $h h'$
for $h' = \sum_{v \in W(A_m)} b_v T_v$
requires
\begin{align*}
  \sum_{v \in W(A_m)} (2 \ell(v) + 1) M
=  \Bigl(M+2\sum_{v \in W(A_m)}  \ell(v)\Bigr) M
= (1+\binom{m+1}{2}) M^2
\end{align*}
operations to produce $M$ coefficient lists $b_v (z_w)_{w\in W} \star T_v$,
which require at most a further $M$ vector additions
at a cost of $M^2$, yielding a total of
\begin{align*}
  (2+\binom{m+1}{2}) M^2
=  \frac{m^2 + m + 4}{2} M^2.
\end{align*}
operations.
\end{proof}

\subsection{Nested coefficient lists}
Now we consider multiplying two elements in $H(A_m)$ given as nested
coefficient lists.
The proof of Theorem~\ref{the:main}(b) requires a definition and two lemmas.

\begin{definition}
For $0 \leq l \leq m$, denote by $c(m, l)$ the
maximum over all $h \in H(A_m)$ and $j$,  $l \leq j \leq m$, of the cost
of computing a product $h T_{a(j,l)}$.
Moreover,
denote by $C(m, l)$ the maximum cost of
computing a product $h g$ over all $h \in H(A_m)$ and $g \in H(A_l)$.
\end{definition}

For $l > m$, we set $c(m, l) = 0$,
as the construction of $h T_{a(j,l)} \in H(A_j)$
from the element $h \in H(A_m)$
does not require any operations in $Z$.

\begin{Lemma}\label{la:cml}
Let $m\geq l\geq 0$ and $M = (m+1)!$.  Then
\begin{enumerate}
\item[(i)]  $c(m,l) \leq \frac{3}{2} l M$;
\item[(ii)] $\sum_{i=1}^lc(m, i) \leq \frac34 l(l+1) M$.
\end{enumerate}
\end{Lemma}

\begin{proof}
(ii) is an obvious consequence of (i).

For (i), we first show that $c(m, l)$ satisfies the following recursion.
\begin{align}\label{recursion}
  c(m, l) \leq (m-l)\, c(m{-}1, l) + (2l+1)\, m! + 2\sum_{i=1}^{l-1}c(m{-}1, i).
\end{align}

Clearly, $c(m, 0) = 0$ since $T_{a(j,0)} =
1_H$.  For $m \geq l > 0$, we can determine $c(m, l)$ recursively as follows.
Let $h = \sum_{k=0}^m h_k T_{a(m,k)}$ for elements $h_k \in H(A_{m-1})$ and
suppose that $h T_{a(j,l)} = \sum_{k=0}^m h_k' T_{a(m,k)}$ for
elements $h_k' \in H(A_{m-1}).$
We consider each of the cases of Proposition~\ref{pro:HA}.

Cases (a) and (d) only occur for $l < m$.
Each value of $k$ which arises in these cases contributes at most $c(m-1,l)$
to the cost of computing $h_k'$.
Thus the total
cost arising from case (a) is at most $(m-j)\, c(m-1,l)$, and the total cost arising
from case (d) is at most $(j-l)\, c(m-1,l)$.
Therefore the total cost arising from cases (a) and (d) together is at most
$(m-l)\,c(m-1,l).$

In case (c), the $i$th summand contributes a cost of
$c(m-1,i-1)$, giving a total contribution of at most
$\sum_{i=1}^l c(m-1,i-1)$ towards the cost of computing the element
$\sum_{i=1}^l h_{m-j+i}T_{a(j-1,i-1)}$ in $H(A_{m-1})$.
Forming the sum requires $l-1$ additions at cost of at most
$(l-1)m!$ operations.
Multiplying this quantity by $q-1$ requires a further
$m!$ operations, and adding $h_{m-j}$ costs an additional $m!$ operations.
Thus the total cost for case (c) is at most
$(l+1)\,m! + \sum_{i=1}^l c(m-1, i-1)$.

The computation required for case (b) can make use of the computation
done for case (c), if we record the
elements   $h_{m-j+i} T_{a(j-1, i-1)}$.
For $1 \leq i \leq l$, set $k=i-1 +m-j$ so that
$m-j+i=k+1$
and $m-j \leq k < m-j+l$.  Thus in
the course of the computations for case (c) we have already obtained the terms
 $h_{k+1} T_{a(j-1, k-m+j)}$ for $m-j \le k < m-j+l$.
For each of these values of $k$, we have, using the second case of Lemma~\ref{la:HA}:
$$  h_{k+1} T_{a(j-1, l-1)} = (h_{k+1} T_{a(j-1, k+j-m)}) T_{a(m-k-1,
  l-1-k+m-j)},$$
which is one of the terms needed for the computation in case (b).
The cost of
multiplying $h_{k+1} T_{a(j-1, k+j-m)}$ by $ T_{a(m-k-1,
  l-1-k+m-j)}$ is $c(m-1, l-1-k+m-j)$.
Thus the total cost of computing the terms $h_{k+1}T_{a(j-1,l-1)}$
for $m-j \leq k < m-j+l$ is at most
$$\sum_{k=m-j}^{m-j+l-1}   c(m-1, l-1-k+m-j) =
   \sum_{i=1}^l  c(m-1, l-i)=  \sum_{i=1}^l  c(m-1, i-1).$$
Multiplying each of the $l$ terms by $q$ requires another $l\,m!$ operations.
Thus the total cost for case (b) is at most
$l\,m! + \sum_{i=1}^l  c(m-1, i-1)$.

Adding the four components, we find that the total cost $c(m, l)$
satisfies
\begin{align*}
    c(m, l) \leq (m-l)\, c(m{-}1, l) + (2l+1)\, m! + 2\sum_{i=1}^lc(m{-}1, i{-}1),
\end{align*}
if $l \leq m$.

Using the fact that  $c(m{-}1, 0) = 0$, we have $\sum_{i=1}^lc(m{-}1, i{-}1) = \sum_{i=1}^{l-1}c(m{-}1, i)$,
and hence the recursion~\eqref{recursion} holds.

For $l > 0$, we  now prove by induction on $m$ that the statement
\begin{align*}\tag{$*$}\label{statement}
  c(m, l) \leq \tfrac32 l (m+1)! \text{ for all } l \text{ satisfying }
  0 < l \leq m
\end{align*}
holds for all $m \geq 0$.

For $m = 0$ there is nothing to prove.

Assume now that $k \geq 1$ and that assertion~\eqref{statement} holds for
$m = k-1$.
Then we have $c(k{-}1,l)\leq \frac{3}{2} l\,k!$ for $l < k$, and
$c(k{-}1,l) = 0$ for $l = k$.
Moreover, using the recursion \eqref{recursion} and part (ii) for $m = k-1$,
\begin{align*}
  c(k,l) &= (2l+1) k! + (k-l)\, c(k{-}1, l) + 2\sum_{i=1}^{l-1} c(k{-}1,i)\\
&\le ((2l+1) + \tfrac{3}{2}(k-l)\, l +\tfrac{3}{2} l(l-1)) k!\\
&= ((2l+1) + \tfrac{3}{2}(k-1)\, l) k!\\
&\le (3l + \tfrac{3}{2}(k-1)\, l) k!\text,
\end{align*}
as $l \ge 1$, and we find that
$c(k,l) \le  \frac32 (k+1) l k! = \frac32 l(k+1)!= \frac32 l (k{+}1)!$,
as desired.

Hence $c(m,l) \leq \frac{3}{2} l (m{+}1)! = \frac{3}{2} l M$ for all
$m \geq l \geq 0$.
\end{proof}

\begin{Lemma}\label{la:Cml} For $0 \leq l \leq m$, we have
  $C(m,l) \le (1 + e) (l+1)!\, M$.
\end{Lemma}

\begin{proof}
If $l = 0$ then $g \in H(A_0) = Z$ is a scalar,
and multiplying $h$ by a scalar costs $C(m, 0) = M$  operations
in the worst case.  It is easy to see that $C(m, 1) \leq 3 M + 3 \leq 2!\cdot 3 \cdot M$
and that $C(m, 2) \leq \tfrac{31}2 M + 9 \leq 3!\cdot 3 \cdot M$.
Hence the claim is true for $l=0,1,2$.

If $l > 2$ then $g = \sum_{k=0}^l g_k T_{a(l, k)} \in H(A_l)$ (for some $g_k \in H(A_{l-1})$) and, using Lemma~\ref{la:cml}, the cost of computing
$hg = \sum_{k=0}^l (hg_k) T_{a(l, k)}$,
where, for each $k$, we first compute $hg_k$ and then multiply the result by
$T_{a(l, k)}$ and then add the $l+1$ terms,
satisfies
\begin{align}\label{recursion2}
  C(m, l) &\leq (l+1) C(m, l-1) + \sum_{k=0}^l c(m, k) + l M \\\notag
&\leq (l+1) (C(m, l-1) + (\tfrac34 + \tfrac{1}{l+1}) l M) \\\notag
&\leq (l+1) (C(m, l-1) + l M) \text.
\end{align}

Define a function $f$ by $f(0)=1$ and
$f(l) = 1 + \sum_{j=0}^{l-1} \frac{1}{j!}$ for $l \ge 1$.

We show by induction on $l$ that
$C(m,l) \le f(l)\,(l+1)!\, M$ for $l > 1$.
For $l = 2$ this follows from the above bound on $C(m, 2)$.

Now suppose that $l > 2$
and that the assertion holds for $l-1$.
Then, by~\eqref{recursion2} and the inductive hypothesis,
\begin{align*}
  C(m, l) &\le (l+1) (C(m, l-1) + l M)\\
&\le (l+1)(f(l{-}1)\,l!\, M + lM)\\
&=  (f(l{-}1) + \tfrac1{(l-1)!})\, (l+1)!\, M\\
&= f(l)\,(l+1)!\,  M\text.
\end{align*}
Since $f(l) = 1 + \sum_{j=0}^{l-1} \frac{1}{j!}\le 1 + \sum_{j=0}^{\infty}
\frac{1}{j!} = 1 + e$, the result follows.
\end{proof}

The proof of Theorem~\ref{the:main}(b) now follows as $C(m,m) \le
(1 + e) M^2$.

\begin{Remark}
  The complexity analysis shows that computing in Iwahori-Hecke
  algebras of type $A$ with nested coefficient lists is more efficient
  than with simple coefficient lists.  Experiments with a prototype
  implementation in GAP4~\cite{GAP4} suggest that the speed-up achieved in
  practice is even more dramatic than the complexity analysis
  predicts.  This may be due to the fact that certain costs arising
  with simple coefficient lists have not been taken into account.  For
  example, when computing the coefficients $z_w'$ in
  Equation~(\ref{eq:starTs}), additional cost can arise from computing
  the product $ws$ in $W$, from comparing the lengths of $w$ and $ws$, and
  from locating the coefficient $z_{ws}$ in a simple coefficient list.
\end{Remark}

\subsection{Acknowledgements}
The first and third authors acknowledge the support of
the Australian Research Council Discovery Project DP140100416.
The second author thanks the School of Mathematics and Statistics at
the University of Western Australia for their hospitality during a visit
to Perth.  We thank the anonymous referees for their useful comments
and suggestions.

\bibliography{fasta}
\bibliographystyle{amsplain}

\end{document}